\documentclass[12pt]{article}

\usepackage{amsmath,amssymb}
\usepackage{tikz}
\usepackage{color}
\usepackage[toc]{appendix}
\usepackage{graphicx}
\usepackage{fancyhdr}
\usepackage{enumitem}
\usepackage{bbm}
\usepackage{graphicx}
\usepackage{parskip}
\usepackage{float}
\usepackage{chngpage}
\usepackage{calc}
\usepackage{bigints}
\usepackage{tikz}
\usepackage{hyperref}
\usepackage{array}
\usepackage{bbm}
\usepackage{booktabs}
\usepackage{rotating}
\usepackage{multirow}
\usepackage{adjustbox}
\usepackage{tabularx}
\usepackage{enumitem}
\usepackage{verbatim}
\usepackage{mathtools}
\usepackage{ragged2e}
\usepackage[makeroom]{cancel}
\usepackage{enumitem}
\usepackage{caption}
\usepackage{mathtools}
\usepackage{bbm}
\usepackage{amsmath,amsfonts,amsthm,bm}
\usepackage{amsmath,amssymb}
\usepackage{appendix}
\usepackage{graphicx}
\usepackage{tikz}
\usepackage{verbatim}
\usepackage{amsthm}
\usepackage[english]{babel}
\usepackage{hyphenat}
\usepackage[makeindex]{imakeidx}
\usepackage{tikz}
\usepackage[switch]{lineno}

\usetikzlibrary{datavisualization}
\usetikzlibrary{matrix}
\usetikzlibrary{datavisualization.formats.functions}

\newtheorem{theorem}{Theorem}[section]

\newtheorem{lemma}[theorem]{Lemma}
\newtheorem{definition}[theorem]{Definition}

\newtheorem{remark}[theorem]{Remark}
\newtheorem{assumption}[theorem]{Assumption}

\makeatletter
\def\section{\@startsection {section}{1}{\z@}{3.25ex plus 1ex minus
		.2ex}{1.5ex plus .2ex}{\large\bf}}
\def\subsection{\@startsection{subsection}{2}{\z@}{3.25ex plus 1ex minus
		.2ex}{1.5ex plus .2ex}{\normalsize\bf}}
\@addtoreset{equation}{section} 
\makeatother
\makeatletter
\newsavebox{\@brx}
\newcommand{\llangle}[1][]{\savebox{\@brx}{\(\m@th{#1\langle}\)}%
	\mathopen{\copy\@brx\kern-0.5\wd\@brx\usebox{\@brx}}}
\newcommand{\rrangle}[1][]{\savebox{\@brx}{\(\m@th{#1\rangle}\)}%
	\mathclose{\copy\@brx\kern-0.5\wd\@brx\usebox{\@brx}}}
\makeatother

\title{Wong–Zakai approximations for quasilinear systems of
	It\^o’s type stochastic differential equations driven by fBm with $H>\frac{1}{2}$}

\author{ Ramiro Scorolli\thanks{Dipartimento di Scienze Statistiche Paolo Fortunati, Università di Bologna, Bologna, Italy. \textbf{e-mail}: ramiro.scorolli2@unibo.it}}

\date{\today}

\begin{document}
	
	\maketitle
	
	\bigskip

	\begin{abstract}
		In a recent article  Lanconelli and Scorolli (2021) extended to the multidimensional case a Wong-Zakai-type approximation for It\^o stochastic differential equations proposed by \O ksendal and Hu (1996). The aim of the current paper is to extend the latter result to system of stochastic differential equations of It\^o type driven by fractional Brownian motion like those considered by Hu (2018).
		
		The covariance structure of the fractional Brownian motion (fBm) precludes us from using the same approach as that used by Lanconelli and Scorolli and instead we employ a \textit{truncated} Cameron-Martin expansion as the approximation for the fBm.
		We are naturally led to the investigation of
		a semilinear hyperbolic system of evolution equations in several space variables that we utilize for constructing a solution
		of the Wong–Zakai approximated systems.
		We show that the law of each element of the approximating
		sequence solves in the sense of distribution a Fokker–Planck equation and that the sequence converges
		to the solution of the Itô equation, as the number of terms in the expansion goes to infinite.
	\end{abstract}
	
	
	
	\bigskip	
	
\section{Introduction and statements of the main results.}	

Let $T>0$ be some arbitrary real number, $\{W_t\}_{t\in[0,T]}$ be a standard one dimensional Brownian motion and for $\epsilon>0$ denote with $\{W_t^{\epsilon}\}_{t\in[0,T]}$  some smooth stochastic process converging to the latter as $\epsilon\to 0$.

Under suitable conditions on the coefficients $b:[0,T]\times\mathbb R\to\mathbb R$ and $\sigma:[0,T]\times\mathbb R\to\mathbb R$ the solution $\{Y^{\epsilon}(t)\}_{t\in [0,T]}$ of the random ordinary differential equation
\begin{align}\label{eq: WZ strat}
\frac{d Y^{\epsilon}(t)}{dt}=b(t,Y^{\epsilon}(t))+\sigma(t,Y^{\epsilon}(t))\cdot\frac{d W^{\epsilon}(t)}{dt}
\end{align}
converges as $\epsilon\to 0$ to the strong solution of the Stratonovich stochastic differential equation (SDE for short)	
\begin{align}\label{eq: SDE strat}
d Y(t)=b(t,Y(t))dt+\sigma(t,Y(t))\circ dW_t,
\end{align}
or equivalently of the It\^o SDE
\begin{align*}
	d Y(t)=\left[b(t,Y(t))+\frac{1}{2}\sigma(t,Y(t))\partial_y\sigma(t,Y(t))\right]dt+\sigma(t,Y(t))dW_t.
\end{align*}
This is the statement of the famous Wong-Zakai theorem \cite{wong1965convergence} whose multidimensional version can be found in \cite{stroock2020support}.

In \cite{hu1996wick} \O ksendal and Hu suggested how to modify equation \eqref{eq: WZ strat} to get in the limit the It\^o's interpretation of \eqref{eq: SDE strat}; they considered the case with $\sigma(t,x)=\sigma(t)x$, where $\sigma:[0,T]\to\mathbb{R}$ is a deterministic function, and proved that the solution $\{X^{\pi}(t)\}_{t\in [0,T]}$ of the differential equation  
\begin{align*}
	\frac{dX^{\epsilon}(t)}{dt}=b(t,X^{\epsilon}(t))+\sigma(t)X^{\epsilon}(t)\diamond\frac{d{W}^{\epsilon}(t)}{dt},
\end{align*}
converges, as $\epsilon$ tends to zero, to the strong solution $\{X(t)\}_{t\in [0,T]}$ of the It\^o SDE
\begin{align}\label{ito 1D}
	dX(t)=b(t,X(t))dt+\sigma(t)X(t)dW(t).
\end{align}
Here, the symbol $X^{\epsilon}(t)\diamond\frac{d{W}^{\epsilon}(t)}{dt}$ stands for the \textit{Wick product} between $X^{\pi}(t)$ and $\frac{d{W}^{\epsilon}(t)}{dt}$. (We postpone to the next section all the necessary mathematical details for the tools utilized in this introduction).
In \cite{LanconelliScorolli2021} Lanconelli and Scorolli presented a multidimensional version of this result in which the linear diffusion matrix diagonal and the $d$-dimensional Brownian motion
is approximated by a piecewise interpolation also known as \textit{polygonal approximation}; just as in \cite{hu1996wick} the product between diffusion coefficients and smoothed white noise is interpreted as a \textit{Wick product}.

The aim of this article is to extend the aforementioned result to the case of a quasilinear system of SDEs driven by a fractional Brownian motion (fBm for short) with Hurst paramenter $H>1/2$ as those considered by Hu in \cite{hu2018ito}.
In order to do that we introduce the following Cauchy problem 
\begin{align}\label{eq: WZ}
\begin{cases}
\frac{d X^K_i(t)}{dt}=b_i(t,X^K(t))+\sigma_i(t)X_i^K(t)\diamond\frac{d B_i^K(t)}{dt},\; t\in ]0,T]\\
X_i(0)=c_i\in\mathbb R,\quad \text{for }\; i\in\{1,...,d\},
\end{cases}
\end{align}
where $\{B^K(t)\}_{t\in [0,T]}$ stands for the\textit{ truncated} (up to the $K$-th term) Cameron-Martin expansion of the $d$-dimensional fractional Brownian motion $\{B(t)\}_{t\in[0,T]}$(the reason for which we have to employ a different approximation from that used in \cite{LanconelliScorolli2021} will be addressed latter on), the coefficients $b:[0,T]\times \mathbb R^d\to\mathbb R^d$, $\sigma:[0,T]\to\mathbb R^d$ satisfy certain condition which will be specified in the following, while $c\in\mathbb R^d$ is a deterministic initial condition.
In order to ease the notation we have omitted (and will do so for the rest of the article) the Hurst parameter $H>1/2$. 

We must interpret \eqref{eq: WZ} as Wong-Zakai approximation of the stochastic Cauchy problem of the It\^o type:
\begin{align}\label{eq:SDE}
	\begin{cases}
		dX_i(t)=b_i(t,X(t))dt+\sigma_i(t)X_i(t)dB_i(t),\; t\in [0,T]\\
		X_i(0)=c_i\in\mathbb R,\; \text{for}\; i\in\{1,...,d\}.
	\end{cases}
\end{align}

It's important to stress the  fact that \eqref{eq: WZ} is not a system of random ordinary differential equations, but rather an evolution equation involving an infinite dimensional gradient (see for instance  \cite[equation 1.5]{hu2018ito}).

All throughout this article we will assume that the coefficients $b$ and $\sigma$ posses enough regularity to ensure that the Cauchy problem \eqref{eq:SDE} has a unique strong solution (e.g. \cite{hu2018ito}).
\begin{assumption}\label{assumptions}\quad
	\begin{itemize}
		\item The functions $b(t,x)$, $\partial_{x_1}b(t,x)$,..., $\partial_{x_d}b(t,x)$ are bounded and continuous; 
		\item the functions $\sigma_1(t),...,\sigma_d(t)$ are bounded and continuous.
	\end{itemize}
\end{assumption}

We are now ready to state our main results; the first of which ensures the existence of a \textit{solution} for the approximating equation.
\begin{theorem} [Existence]\label{th:existence}
Let Assumption \ref{assumptions} be in force. Then \eqref{eq: WZ} has a mild solution in the sense of definition \ref{def: solution}.
\end{theorem}
	
Our second theorem states that the law of the \textit{approximation} solves a Fokker-Planck-like equation 
\begin{theorem} [Fokker-Planck equation]\label{th:fokker planck}
The law
\begin{align*}
\mu^{K}(t,A):=\mathbb P(\{\omega\in\Omega:X^K(t,\omega)\in A\}), \quad t\in [0,T], A\in\mathcal B(\mathbb R^d)
\end{align*}
of the random vector $X^{K}(t)$ solves in the sense of distributions the Fokker-Planck equation
\begin{align}
\left(\partial_t+\sum_{i,j=1}^d\sum_{k=1}^K\sigma_i(t)\xi_k(t)x_ig_{ik}^{(j)}(t,x)\partial^2_{x_ix_j}+\sum_{i=1}b_i(t,x)\partial_{x_i}\right)^*u(t,x)=0,\quad (t,x)\in[0,T]\times\mathbb R^d
\end{align}
where $g_{ik}^{(j)}:[0,T]\times\mathbb R^d\to\mathbb R$ is a measurable function defined in \eqref{eq:g}.
\end{theorem}

The third and last theorem states that the \textit{approximation} indeed converges to the strong solution of the It\^o SDE;
\begin{theorem}[Convergence]\label{main theorem 3}\label{th: convergence}
	The mild solution $\{X^{K}(t)\}_{t\in [0,T]}$ converges as  $K$ tends to infinite, to the unique strong solution $\{X(t)\}_{\in [0,T]}$ of the It\^o SDE \eqref{eq:SDE}. More precisely,
	\begin{align*}
		\lim_{K\to\infty}\sum_{i=1}^d\mathbb{E}\left[\left|X_i^{K}(t)-X_i(t)\right|\right]=0,\quad\text{ for all }t\in [0,T].
	\end{align*}
\end{theorem}

\section{Preliminaries}
\subsection{Elements on fractional Brownian motion.}
In this section we will introduce the basic concepts that will be needed in order to prove our results. For further details the interested reader is referred to the excellent references \cite{hu2016analysis}\cite{biagini2008stochastic}\cite{duncan2000stochastic}.

Start by fixing $H\in\left(\frac{1}{2},1\right)$, and let $\Omega:=C_0([0,T],\mathbb R^d)$ be the space of $\mathbb R^d$-valued continuous functions endowed with the topology of uniform convergence. There is a probability measure $P^H$ on $(\Omega,\mathcal B(\Omega))$, such that on $(\Omega,\mathcal B(\Omega),P^H)$ the coordinate process $B:\Omega\to\mathbb R^d$ defined as
\begin{align*}
B(t,\omega)=\omega(t),\quad \omega\in\Omega
\end{align*}
is a $d$-dimensional fBm, i.e. a $d$-dimensional centered Gaussian stochastic process in which for each $i\in \{1,...,d\}$ it holds that
\begin{align*}
\mathbb E\left[B_i(t)B_i(s)\right]=\frac{1}{2}\left(t^{2H}+s^{2H}- |t-s|^{2H}\right),\;s,t\in[0,T].
\end{align*}

Let 
\begin{align}\label{eq:phi}
\phi(s,t):=H(2H-1)|s-t|^{2H-2},\;s,t\in[0,T],
\end{align}
and define
\begin{align*}
\mathcal H_{\phi}:=\bigg\{f:[0,T]\to\mathbb R: |f|_{\phi}^2=\int_0^T\int_0^T f(s)f(t)\phi(s,t)dsdt<\infty\bigg\}.
\end{align*}
If $\mathcal H_{\phi}$ is equipped with the inner product 
\begin{align*}
\langle f,g\rangle_{\phi}=\int_0^T\int_0^T f(s)g(t)\phi(s,t)dsdt,
\end{align*}
then it becomes a separable Hilbert space, moreover we can see that $\mathcal H_{\phi}$ equals the closure of $L^2([0,T])$ with respect to the inner product $\langle\cdot,\cdot\rangle_{\phi}$. 
For $f\in\mathcal H_{\phi}$ we denote with $\Phi[f]:[0,T]\to\mathbb R$ the following continuous map
\begin{align*}
&[0,T]\ni t\to \mathbb R,\\
&t\mapsto \int_0^T f(s)\phi(t,s)ds.
\end{align*}

For a deterministic function $f\in\mathcal H_{\phi}$ we can define in the usual manner a \textit{fractional} Wiener integral satisfying the following isometry property 
\begin{align}\label{eq:isometry}
\mathbb E\left[\left(\int_0^T f(s)dB_i(s)\right)^2 \right]=|f|_{\phi}^2.
\end{align}
For $f\in\mathcal H_{\phi}$ and $i\in\{1,...,d\}$ define the \textit{stochastic exponential of }$f$ by
\begin{align*}
\mathcal E_i(f):=\exp\bigg\{\int_0^Tf(s)B_i(s)-\frac{1}{2}|f|_{\phi}^2\bigg\}.
\end{align*}
then the linear span of $\big\{\mathcal E_i(f); f\in\mathcal H_{\phi},\; i\in \{1,...,d\}\big\}$ is dense in $\mathbb L^2(\Omega)$.
\subsection{Approximating equation.}
The first step when constructing a Wong-Zakai approximation is to choose a sequence of smooth stochastic processes converging to the Brownian motion that drives the original equation. 
\begin{remark}
In \cite{LanconelliScorolli2021} we have employed the so called \textit{polygonal} approximation; however the non-independence of the increments of the fBm precludes us from using this approach (see \cite[Remark 3.2]{LanconelliScorolli2021}). 
\end{remark}

 Instead let's  assume that $\{e_k\}_{k\geq 1}$ is a complete orthonormal system (CONS) of the Hilbert space $\mathcal H_{\phi}$, then (e.g. \cite[equation 3.21]{hu2003fractional} ) the $d$-dimensional fractional Brownian motion has the following Cameron-Martin expansion
\begin{align*}
B_i(t)=\sum_{k=1}^{\infty}\left[\int_0^t\left(\int_0^T e_k(r)\phi(v,r)dr\right)dv\right]\int_0^Te_k(s)dB_i(s),\,t\in[0,T],\;\text{for }i\in\{1,...,d\}.
\end{align*}
From this expression it's then straightforward to see that a natural approximation for the fractional \textit{white noise} is given by
	\begin{align*}
		\frac{d B_i^K(t)}{dt}:=\sum_{k=1}^{K}\left(\int_0^T e_k(r)\phi(t,r)dr\right)\int_0^Te_k(s)dB_i(s), \,t\in[0,T],\;\text{for }i\in\{1,...,d\};
	\end{align*}
i.e. the time derivative of the \textit{ truncated} Cameron-Martin expansion. The convergence of this object to the ``singular fractional white noise'' as $K$ goes to infinity must be understood in a space of generalized random variables (see \cite{hu2003fractional} for further details).
	
Notice that due to \eqref{eq:isometry}  and the orthonormality of $\{e_k\}_{k\geq 1}$ if we let $Z_k^{(i)}:=\int_0^Te_k(s)dB_i^H(s)$, then $\left(Z_k^{(i)}\right)_{(k,i)\in\{1,...,K\}\times\{1,...,d\}}$  is a family of i.i.d. standard Gaussian random variables.
	
For the ease of notation let $\xi_k(\cdot)=\Phi [e_k](\cdot)$ for all $k\in\{1,...,K\}$ 
and hence our approximation for  the fractional white noise can be written as
\begin{align*}
		\frac{d B_i^K(t)}{dt}=\sum_{k=1}^{K}\xi_k(t)Z_{k}^{(i)},\,t\in[0,T],\;\text{for }i\in\{1,...,d\}.
\end{align*}

One of the key tools that will be employed in this article is the so called \textit{Wick product} which can be defined for any couple of random variables $X$ and $Y$ belonging to $\mathbb L^p(\Omega)$ for some $p>1$, (e.g. \cite{holden1996stochastic},\cite{hu2016analysis},\cite{janson1997gaussian}). For our purposes it's enough to consider just a few particular cases:

\begin{itemize}
\item if $X\in\mathbb L^p(\Omega)$ for some $p>1$ and $f\in\mathcal H_{\phi}$ we set 
	\begin{align}\label{eq:gjessing}
		X\diamond \mathcal E_i(f):=\mathtt T_{-\Phi[f]_i}X\cdot  \mathcal E_i(f)
	\end{align}
	where $\mathtt T_{-\Phi [f]_i}$ stands for the \textit{translation} operator 
	\begin{align}
		(\mathtt T_{-\Phi [f]_i} X)(\omega):=X\left(\omega-\epsilon_i\int_0^{\cdot}\Phi[f](r)dr\right).
	\end{align}
	Here $\{\epsilon_1,...,\epsilon_d\}$ denotes the canonical basis of $\mathbb R^d$. This is the fractional analog of the so called Gjessing lemma (\cite{holden1996stochastic} Theorem $2.10.6$ ).
	From the latter we are able to see that the \textit{Wick product} with a \textit{stochastic exponential} preserves the monotonicity, i.e. :
	\begin{align*}
\mbox{if $X\leq Y$ then $X\diamond \mathcal E_i(f)\leq Y\diamond \mathcal E_i(f)$}.
	\end{align*}
\item A consequence of the latter and the density of the \textit{stochastic exponentials }is that if $g\in \mathcal H_{\phi}$, $F\in \mathbb L^p(\Omega)$ and $\langle  D_{\phi}^{(i)}F,g\rangle\in\mathbb L^p(\Omega)$ for some $p>1$ then 
\begin{align}\label{eq:wick malliavin}
F\diamond \int_0^T g(s)dB_i(s)=F\int_0^T g(s)dB_i(s)-\langle  D_{\phi}^{(i)}F,g\rangle
\end{align}
where $D_{\phi}^{(i)}$ denotes the $\phi$-derivative (e.g. \cite{hu2003fractional}\cite{duncan2000stochastic}) with respect to the $i$-th fBm and $\langle \cdot,\cdot\rangle$ denotes the inner product of $\mathbb L^2([0,T])$
\end{itemize}

With all this in hand we are able to provide a solution concept for \eqref{eq: WZ}:

\begin{definition}\label{def: solution}
	A $d$-dimensional stochastic process $\{X^{K}(t)\}_{t\in [0,T]}$ is said to be a \emph{mild} solution of equation \eqref{eq: WZ} if:	
	\begin{enumerate}
		\item the function $t\mapsto X^{K}(t)$ is almost surely continuous;
		\item for all $i\in\{1,...,d\}$ and $t\in [0,T]$, the random variable $X_i^K(t)$ belongs to $\mathbb L^p(\Omega)$ for some $p>1$;
		\item for all $i\in\{1,...,d\}$, the identity
		\begin{align}\label{eq:solution}
			X_i^{K}(t)=c_i\mathcal E_i^{K}(0,t)+\int_0^t b_i(s,X^{K}(s))\diamond \mathcal E_i^{K}(s,t)ds,\quad t\in [0,T],
		\end{align}
		holds almost surely, where for any $t,r\in [0,T], r\leq t$, $\mathcal E_i^K(r,t)$ is a shorthand for $\mathcal E_i(\sigma_i^K(r,t))$ where $\sigma_i^K(r,t;\cdot)$ denotes the orthogonal projection of $\chi_{[r,t]}\sigma_i(\cdot)$ on $\operatorname{span}\{e_1,...,e_K\}\subset \mathcal H_{\phi}$.
		
	\end{enumerate}
\end{definition}

\section{Proof of theorem \ref{th:existence}}
In order to prove the existence of a mild solution for \eqref{eq: WZ} we will introduce a system of partial differential equations which is related to \eqref{eq: WZ} by the following heuristic considerations.
\begin{remark}
Formally applying identity \eqref{eq:wick malliavin} we can rewrite (\ref{eq: WZ}) as 	
\begin{align*}
	\begin{cases}
		\frac{d	X_i^K(t)}{d}=b_i(t,X^K(t))+\sigma_i(t)X_i(t)\cdot \left(\sum_{k=1}^{K}\xi_k(t)Z_{k}^{(i)}\right) -\sum_{k=1}^{K}\sigma_i(t)\xi_k(t)\left\langle D^{(i)}_{\phi}X_i(t),e_k\right\rangle,\\
		t\in ]0,T]\\
		X_i(0)=c_i\in\mathbb R,\; \text{for}\; i\in\{1,...,d\}.
	\end{cases}
\end{align*}
If we now search for a solution of the form 
\begin{align*}
	X_i^K(t,\omega)=u_i(t,\mathbf{z}(\omega)),
\end{align*}
for $u_i:[0,T]\times \mathbb R^{K\times d}\to\mathbb R$ where we identify $\mathbb R^{K\times d}$ with the space of $(K\times d)$ and $\mathbf{z}_{ki}(\omega)=Z_{k}^{(i)}(\omega)$ then by a simple application of the chain rule for the $\phi$-derivative we see that $u=(u_1,...,u_d)$ has to solve the following semilinear hyperbolic system of partial differential equations
\begin{align}\label{eq:PDEu}
	\begin{cases}
		\partial_t u_i=b_i(t,u)+\sigma_i(t)\sum_{k=1}^K\xi_k(t)\left[x_{ki}u_i-\partial_{x_{ki}}u_i\right],\\
		(t,\mathbf{x})\in]0,T]\times\mathbb R^{K\times d},\\
		u_i(0,\mathbf{x})=c_i\in\mathbb R,\; \text{for}\; i\in\{1,...,d\}.
	\end{cases}
\end{align}
\end{remark}
Unfortunately to the best of our knowledge the latter does not satisfy the basic assumption of the main existence-uniqueness theorems present in the literature.
For that reason we will introduce the following auxiliary Cauchy problem
\begin{align}\label{eq:PDEv}
	\begin{cases}
		\partial_t v_i=b_i\left(t,v(t)\exp\big\{\frac{1}{2}\|\mathbf{x}\|_F^2\}\right)\exp\big\{-\frac{1}{2}\|\mathbf{x}\|_F^2\}-\sigma_i(t)\sum_{k=1}^K\xi_k(t)\partial_{x_{ki}}v_i,\\
		 (t,\mathbf{x})\in]0,T]\times\mathbb R^{K\times d},\\
		v_i(0,\mathbf{x})=c_i\exp\big\{-\frac{1}{2}\|\mathbf{x}\|_F^2\},\; \text{for}\; i\in\{1,...,d\},
	\end{cases}
\end{align}
where $\|\cdot\|_F$ denotes the Frobenious norm, i.e. $\|\mathbf{x}\|_F^2:=\sum_{i=1}^d\sum_{k=1}^K|x_{ki}|^2$. Our motivation for doing so will be clear in a moment.

A closer inspection would allow the reader to see that the latter is a $d$-dimensional semilinear symmetric hyperbolic system of evolution equations in $(K\times d)$ spatial variables.
The validity of assumption \ref{assumptions} implies the existence of a unique classical solution of \eqref{eq:PDEv} for any arbitrary time interval $[0,T]$ (see for instance \cite{taylor2013partial},\cite{bressan2000hyperbolic} and \cite{majda2012compressible}).

If we let
\begin{align}
\Sigma_{i,k}(r,t):=\int_r^t\sigma_i(s)\xi_k(s)ds =\langle \chi_{[r,t]}\sigma_i,e_k\rangle_{\phi},
\end{align}
then we can write down a \textit{mild solution} for \eqref{eq:PDEv} as 
\begin{align}\label{eq:mild v}
\begin{cases}
&v_i(t,\mathbf{x})=c_i\exp\big\{-\frac{1}{2}\|\mathbf{x}-\mathbf\Sigma^{(i)}(t)\|_F^2\big\}\\
&\quad\quad+\int_0^t b_i\left(s,v(s,\mathbf{x}-\mathbf\Sigma^{(i)}(s,t))\exp\big\{\frac{1}{2}\|\mathbf{x}-\mathbf\Sigma^{(i)}(s,t)\|_F^2\big\}\right)\exp\big\{-\frac{1}{2}\|\mathbf{x}-\mathbf\Sigma^{(i)}(s,t)\|_F^2\big\}\;ds,\\
&\text{for}\; t\in[0,T],\mathbf{x}\in \mathbb R^{K\times d},\; i\in\{1,...,d\}
\end{cases}
\end{align}
where for any pair $r,t\in [0,T], t\geq r$, we denote with $\mathbf\Sigma^{(i)}(r,t)$ the $(K\times d)$-matrix where the $i$-th column is given by $[\Sigma_{i,1}(r,t),...,\Sigma_{i,K}(r,t)]^T$ and all the remaining components are equal $0$.

Now if we let $u_i(t,\mathbf x):=v_i(t,\mathbf x)\exp\big\{\frac{1}{2}\|\mathbf{x}\|_F^2\big\}$ for all $i\in\{1,...,d\}$ a simple application of the chain rule shows that $u$ solves \eqref{eq:PDEu}, furthermore using \eqref{eq:mild v} we have that the following mild representation holds
\begin{align}\label{eq:mild uu}
\begin{cases}
&u_i(t,\mathbf{x})=c_i\exp\bigg\{\sum_{k=1}^{K}\left[x_{ik}\Sigma_{i,k}(0,t)-\frac{1}{2}|\Sigma_{i,k}(0,t)|^2\right]\bigg\}\\
&\quad\quad\quad+\int_0^t b_i\big(s,u(s,\mathbf{x}-\mathbf\Sigma^{(i)}(s,t))\big)\exp\bigg\{\sum_{k=1}^{K}\left[x_{ik}\Sigma_{i,k}(s,t)-\frac{1}{2}|\Sigma_{i,k}(s,t)|^2\right]\bigg\}\;ds,\\
&\text{for}\; t\in[0,T],\mathbf{x}\in \mathbb R^{K\times d}, i\in\{1,...,d\}.
\end{cases}
\end{align}
\begin{remark}
This equation is the analog of \cite[equation 3.4]{LanconelliScorolli2021}, where intead of shifting the $i$-th component of the vector of spatial variables we shift the $i$-th column of the matrix of spatial variables.
\end{remark}
At this point we define the candidate solution $\{X^K(t)\}_{t\in[0,T]}$ as
\begin{align}\label{eq:solution def}
	X_i^K(t,\omega)=u_i(t,\mathbf{z}(\omega)),\; t\in [0,T],\omega\in \Omega,\; \text{for}\; i\in\{1,...,d\}
\end{align}
where again $\mathbf{z}(\omega)$ is the $K\times d$-matrix in which the $(k,i)$-th component is given by $Z_k^{(i)}(\omega)$.

Next we must verify that $\{X^K(t)\}_{t\in[0,T]}$ is indeed a mild solution of the system \eqref{eq: WZ}, i.e. that satisfies the conditions imposed by definition \ref{def: solution}.

The almost surely continuity of the path is given by the continuity of $[0,T]\ni t\mapsto u(t,\mathbf x):=v(t,\mathbf x)e^{\frac{\|\mathbf x\|_F^2}{2}}$ for all $\mathbf x\in\mathbb R^{K\times d}$ (remember that $v$ is a classical solution).

Noticing that  
\begin{align}\label{eq:wick exp WZ}
	&\exp\bigg\{\sum_{k=1}^{K}\left[x_{ik}\Sigma_{i,k}(r,t)-\frac{1}{2}|\Sigma_{i,k}(r,t)|^2\right]\bigg\}\bigg\rvert_{{x}_{ki}=Z_k^{(i)}(\omega)}\nonumber\\
	&=\exp\bigg\{\int_0^T\sum_{k=1}^K \left(\int_r^t\sigma_i(s)\xi_k(s)ds\right)e_k(q)dB_i^H(q)-\frac{1}{2}\sum_{k=1}^K\left(\int_r^t\sigma_i(s)\xi_k(s)ds\right)^2\bigg\}\nonumber\\
	&=\exp\bigg\{\int_0^T\sum_{k=1}^K \langle \chi_{[r,t]}\sigma_i,e_k\rangle_{\phi}e_k(q)dB_i^H(q)-\frac{1}{2}\sum_{k=1}^K\langle \chi_{[r,t]}\sigma_i,e_k\rangle_{\phi}^2\bigg\}\nonumber\\
	&=\exp\bigg\{\int_0^T\sigma_i^K(r,t;q)dB_i^H(q)-\frac{1}{2}|\sigma_i^K(r,t;\cdot)|_{\phi}^2\bigg\}=:\mathcal E^K_i(r,t),
\end{align}
and using assumption \ref{assumptions} we have that
\begin{align*}
|X_i^K(t)|\leq |c_i|\exp\bigg\{\int_0^T\sigma_i^K(0,t;q)dB_i^H(q)\bigg\}+M\int_0^t\exp\bigg\{\int_0^T\sigma_i^K(s,t;q)dB_i^H(q)\bigg\}ds,
\end{align*}
where $M>0$ is a constant such that $|b(t,\mathbf{x})|\leq M$.
Taking the $\mathbb L^p(\Omega)$-norm on both sides above and using the triangular inequality we obtain
\begin{align*}
	\|X_i^K(t)\|_p&\leq |c_i|\left\|\exp\bigg\{\int_0^T\sigma_i^K(0,t;q)dB_i^H(q)\bigg\}\right\|_p+M\int_0^t\left\|\exp\bigg\{\int_0^T\sigma_i^K(s,t;q)dB_i^H(q)\bigg\}\right\|_pds\\
	&\leq |c_i|\exp\bigg\{\frac p 2|\sigma_i^K(0,t;\cdot)|^2_{\phi}\bigg\}+M\int_0^t\exp\bigg\{\frac p 2|\sigma_i^K(s,t;\cdot)|^2_{\phi}\bigg\}ds\\
	&\leq |c_i|\exp\bigg\{\frac p 2|\sigma_i^K(0,t;\cdot)|^2_{\phi}\bigg\}+Mt\sup_{s\in[0,t]}\exp\bigg\{\frac p 2|\sigma_i^K(s,t;\cdot)|^2_{\phi}\bigg\}<\infty,
\end{align*}
where we used the fact that the fractional Wiener integral of $\sigma_i^K(r,t;\cdot)$ is a Gaussian random variable.
This proves the membership of $X_i^K(t)$ to $\mathbb L^p(\Omega), p\geq 1$ for all $i\in\{1,2,...,d\}$ and $t\in [0,T]$

Last thing we need to do is to prove that the process $X^{K}(t)$ satisfies the representation \eqref{eq:solution}. First we notice that for any $l\in\{1,...,K\}$ and $i\in\{1,...,d\}$
\begin{align*}
	Z_{l}^{(i)}-\Sigma_{i,l}(s,t)=Z_{l}^{(i)}-\sum_{k=1}^{K}\langle\chi_{[s,t]}\sigma_i,e_k\rangle_{\phi}\int_0^T\int_0^Te_k(r)e_l(q)\phi(r,q)drdq,=\mathtt{T}_{-\Phi[\sigma_i^K(s,t)]} Z_{l}^{(i)} 
\end{align*}

where $=\mathtt{T}_{-\Phi[\sigma_i^K(s,t)]}$ is a shorthand for $=\mathtt{T}_{-\Phi[\sigma_i^K(s,t)]_i}$.
Then it follows from \eqref{eq:mild uu} and \eqref{eq:wick exp WZ} that 
\begin{align*}
X_i^K(t)=c_i\mathcal E^K_i(0,t)+\int_0^t \mathtt T_{-\Phi[\sigma_i^{K}(s,t)]}b_i\left(s,X^K(s)\right)\cdot\mathcal E^K_i(s,t)\;ds.\\
\end{align*}
Using identity \eqref{eq:gjessing} we have that for all $t\in [0,T]$ and $i\in\{1,...,d\}$ the following holds a.s.
\begin{align*}
	X^K_i(t)&=c_i\mathcal E^K_i(0,t)+\int_0^t b_i\big(s,X^K(s)\big)\diamond\mathcal E^K_i(s,t)\;ds;
\end{align*}
completing the proof.

\section{Proof of theorem \ref{th: convergence}}

The aim of this section is to show that the \textit{mild solution} of 
\begin{align*}
	\begin{cases}
		\frac{d X^K_i(t)}{dt}=b_i(t,X^K(t))+\sigma_i(t)X_i^K(t)\diamond\frac{d B_i^K(t)}{dt},\; t\in ]0,T]\\
		X_i(0)=c_i\in\mathbb R,\quad \text{for}\; i\in\{1,...,d\};
	\end{cases}
\end{align*}
 converges in $\mathbb L^1(\Omega)$ to the unique strong solution of
\begin{align*}
	\begin{cases}
		dX_i(t)=b_i(t,X(t))dt+\sigma_i(t)X_i(t)dB_i(t),\; t\in [0,T]\\
		X_i(0)=c_i\in\mathbb R,\quad \text{for}\; i\in\{1,...,d\}.
	\end{cases}
\end{align*}
The first thing we must do is to rewrite the solution of the latter stochastic Cauchy problem in a way that resembles identity \eqref{eq:solution}. To that aim we will mimic a reduction method proposed by \cite{hu1996wick} which is based on techniques from the Wick calculus and fractional White noise (see also \cite[theorem 3.6.1]{holden1996stochastic}). We stress the fact that the following steps are somewhat \textit{formal} in our setting but can be performed rigorously under some more technical considerations (see for instance \cite{hu2003fractional} or \cite{biagini2008stochastic}).

Let 
\begin{align*}
\mathtt E_i(0,t):=\exp\bigg\{-\int_0^t\sigma_i(s)dB_i(s)-\frac 1 2|\chi_{[0,t]}\sigma_i|_{\phi}^2\bigg\},
\end{align*}
and 
\begin{align*}
	\mathcal E_i(0,t):=\exp\bigg\{\int_0^t\sigma_i(s)dB_i(s)-\frac 1 2|\chi_{[0,t]}\sigma_i|_{\phi}^2\bigg\}.
\end{align*}

Using equation ($3.41$) of  we can formally write \eqref{eq:SDE} as 
\begin{align*}
\begin{cases}
	\frac {dX_i(t)}{dt}=b_i(t,X(t))+\sigma_i(t)X_i(t)\diamond \frac{d B_i(t)}{dt},\; t\in [0,T]\\
	X_i(0)=c_i,\;\text{for }i\in\{1,...,d\},
\end{cases}
\end{align*}
were we must bare in mind that the time derivative of the fBm is not well defined as a random variable, so in order to make sense of the expression above we must interpret it as a differential equation in some space of generalized random variables (e.g. \cite{biagini2008stochastic}\cite{hu2003fractional}).
Then we can Wick-multiply both sides of the equality above by $\mathtt E_i(0,t)$ which gives, after rearranging
\begin{align*}
	\frac {dX_i(t)}{dt}\diamond \mathtt E_i(0,t) -\sigma_i(t)X_i(t)\diamond \frac{d B_i(t)}{dt}\diamond \mathtt E_i(0,t)=b_i(t,X(t))\diamond \mathtt E_i(0,t).
\end{align*}
By means of the identity
\begin{align*}
	\frac{d \mathtt E_i(0,t)}{dt}=\sigma_i(t)\mathtt E_i(0,t)\diamond \frac{d B_i(t)}{dt}.
\end{align*}  
and the Leibniz rule for the  \textit{Wick product} we obtain
\begin{align}
\frac{d \mathcal X_i(t)}{dt}=b_i(t, X(t))\diamond \mathtt E_i(0,t),
\end{align}
where 
\begin{align*}
\mathcal X_i(t):=X_i(t)\diamond\mathtt E_i(0,t).
\end{align*}
It follows that 
\begin{align*}
\mathcal X_i(t)=X_i(t)\diamond\mathtt E_i(0,t)=c_i+\int_0^t b_i(s,X(s))\diamond \mathtt E_i(0,s)ds
\end{align*}
or which is equivalent
\begin{align*}
	X_i(t)=c_i\mathcal E_i(0,t)+\int_0^t b_i(s,X(s))\diamond \mathcal E_i(s,t)ds
\end{align*}
where we used the identity
\begin{align*}
\mathcal E_i(0,t)\diamond \mathtt E_i(0,t)=1,\; a.s.\; \text{for all } t\in [0,T].
\end{align*}

We conclude that the unique strong solution of \eqref{eq:SDE} satisfies the following integral equation
\begin{align}\label{eq:strong solution}
	X_i(t)=c_i\mathcal E_i(0,t)+\int_0^t b_i(s,X(s))\diamond \mathcal E_i(s,t)ds
\end{align}
for all $t\in [0,T]$ and $i\in\{1,...,d\}$.

\begin{remark}\label{remark}
Under assumption \ref{assumptions} it follows that for any $t\in[0,T]$ the strong solution of \eqref{eq:SDE}  belongs to $\mathbb L^p(\Omega)$ for any $p\geq 1$.
\end{remark}

Now we are ready to prove the convergence;
\begin{align*}
|X_i(t)-X_i^K(t)|&\leq |c_i||\mathcal E_i(0,t)-\mathcal E_i^K(0,t)|\\
&+\int_0^t|b_i(s,X(s))\diamond \mathcal E_i(s,t)-b_i(s,X^K(s))\diamond \mathcal E_i^K(s,t)|ds\\
&=|c_i||\mathcal E_i(0,t)-\mathcal E_i^K(0,t)|\\
&+\int_0^t\big|b_i(s,X(s))\diamond \mathcal E_i(s,t)-b_i(s,X(s))\diamond \mathcal E_i^K(s,t)\\
&+b_i(s,X(s))\diamond \mathcal E_i^K(s,t)-b_i(s,X^K(s))\diamond \mathcal E_i^K(s,t)\big|ds\\
&\leq|c_i||\mathcal E_i(0,t)-\mathcal E_i^K(0,t)|\\
&+\int_0^t\left|b_i(s,X(s))\diamond \left(\mathcal E_i(s,t)- \mathcal E_i^K(s,t)\right)\right|ds\\
&+\int_0^t\left|b_i(s,X(s))-b_i(s,X^K(s))\right|\diamond \mathcal E_i^K(s,t)ds.\\
\end{align*}
Using the Lipschitz continuity of $b_i$ and the fact that the \textit{Wick product} with a \textit{stochastic exponential} preserves the monotonicity we have that
\begin{align*}
|X_i(t)-X_i^K(t)|&\leq |c_i||\mathcal E_i(0,t)-\mathcal E_i^K(0,t)|+\int_0^t\left|b_i(s,X(s))\diamond \left(\mathcal E_i(s,t)- \mathcal E_i^K(s,t)\right)\right|ds\\
&+L \int_0^t \sum_{j=1}^{d} |X_j(s)-X_j^K(s)|\diamond \mathcal E_i^K(s,t)ds\\
\end{align*}
where $L$ is a positive constant such that for all $t\in [0,T]$ it holds $|b_i(t,X)-b_i(t,Y)|\leq L|X-Y|_1$; here $|\cdot|_1$ denotes the $\ell^1$ norm.
Now we take expectation yielding
\begin{align*}
	\mathbb E \left[|X_i(t)-X_i^K(t)|\right]&\leq |c_i|\mathbb E \left[|\mathcal E_i(0,t)-\mathcal E_i^K(0,t)|\right]+\int_0^t\mathbb E \left[\left|b_i(s,X(s))\diamond \left(\mathcal E_i(s,t)- \mathcal E_i^K(s,t)\right)\right|\right]ds\\
	&+L \int_0^t \sum_{j=1}^{d} \mathbb E\left[|X_j(s)-X_j^K(s)|\right]ds.
\end{align*}
The previous inequality is valid for all $i = 1, . . . , d$ and $t\in [0,T]$; therefore, summing over $i$ and setting
\begin{align*}
\mathtt X^K(t):=\sum_{i=1}^{d} \mathbb E\left[|X_j(t)-X_j^K(t)|\right]
\end{align*}
we obtain
\begin{align*}
\mathtt X^K(t)&\leq\sum_{i=1}^{d}  |c_i|\mathbb E \left[|\mathcal E_i(0,t)-\mathcal E_i^K(0,t)|\right]+\sum_{i=1}^{d} \int_0^t\mathbb E \left[\left|b_i(s,X(s))\diamond \left(\mathcal E_i(s,t)- \mathcal E_i^K(s,t)\right)\right|\right]ds\\
&+Ld\int_0^t \mathtt X^K(s)ds\\
&=\mathcal M^K(t)+Ld\int_0^t \mathtt X^K(s)ds,
\end{align*}
where 
\begin{align*}
\mathcal M^K(t):=\sum_{i=1}^{d}  |c_i|\mathbb E \left[|\mathcal E_i(0,t)-\mathcal E_i^K(0,t)|\right]+\sum_{i=1}^{d} \int_0^t\mathbb E \left[\left|b_i(s,X(s))\diamond \left(\mathcal E_i(s,t)- \mathcal E_i^K(s,t)\right)\right|\right]ds.
\end{align*}
According to Gronwall's inequality  the previous estimate yields
\begin{align}\label{eq:gronwall}
\mathtt X^K(t)\leq \mathcal M^K(t)+Ld\int_0^t\mathcal M^K(s)e^{Ld(t-s)}ds;
\end{align}
and hence the proof will be complete if we show that $\mathcal M^K(t)$ is bounded for all $t\in [0,T]$ and it holds that 
\begin{align*}
\lim_{K\to\infty} \mathcal M^K(t)=0,\quad \text{for all}\; t\in [0,T];
\end{align*}
this will allow us to use dominated convergence for the Lebesgue integral appearing in \eqref{eq:gronwall} and conclude that 
\begin{align*}
\lim_{K\to\infty}\mathtt X^K(t)=0.
\end{align*}

In order to prove the boundedness we write
\begin{align*}
\mathcal M^K(t)&\leq \sum_{i=1}^d |c_i|\left(\mathbb E \left[|\mathcal E_i(0,t)\right]+\mathbb E \left[\mathcal E_i^K(0,t)|\right]\right)\\
&+\int_0^t\mathbb E \left[\left|b_i(s,X(s))\diamond \mathcal E_i(s,t)- b_i(s,X(s))\diamond\mathcal E_i^K(s,t)\right|\right]ds\\
&\leq 2\sum_{i=1}^d |c_i|+\sum_{i=1}^d\int_0^t\mathbb E \left[\left|b_i(s,X(s))\diamond \mathcal E_i(s,t)\right|\right]+\mathbb E \left[\left|b_i(s,X(s))\diamond \mathcal E_i^K(s,t)\right|\right] ds\\
&\leq  2\sum_{i=1}^d |c_i|+2dMt;
\end{align*}
the boundedness also follows from the continuity of $t\mapsto\mathcal M^K(t)$ and the compactness of $[0,T]$.

Thus it suffices to prove that
\begin{align*}
	\lim_{K\to\infty} \mathcal M^K(t)=0,\quad \text{for all}\; t\in [0,T].
\end{align*}

Using the fact that $\mathcal E_i^K(0,t)$ converges in $\mathbb L^p(\Omega), p\geq 1$ to $\mathcal E_i(0,t)$ (see Appendix A) it follows that
\begin{align*}
	\lim_{K\to\infty}\mathcal{M}^{K}(t)=&\lim_{K\to\infty}\sum_{i=1}^d|c_i|\mathbb{E}\left[\left|\mathcal E_i^{K}(0,t)-\mathcal E_i(0,t)\right|\right]\\
	&+\lim_{K\to\infty}\sum_{i=1}^d\int_0^t\mathbb{E}\left[\left|b_i(s,X(s))\diamond\left(\mathcal E_i^{K}(s,t)-\mathcal E_i(s,t)\right)\right|\right]ds\\
	=&\sum_{i=1}^d\lim_{K\to\infty}\int_0^t\mathbb{E}\left[\left|b_i(s,X(s))\diamond\left(\mathcal E_i^{K}(s,t)-\mathcal E_i(s,t)\right)\right|\right]ds.
\end{align*}

We now prove that we can take the last limit inside the integral; first of all, note that the integrand is bounded: in fact,
\begin{align*}
	\mathbb{E}\left[\left|b_i(s,X(s))\diamond\left(\mathcal E_i^{K}(s,t)-\mathcal E_i(s,t)\right)\right|\right]=&\mathbb{E}\left[\left|b_i(s,X(s))\diamond\mathcal E_i^{K}(s,t)-b_i(s,X(s))\diamond\mathcal E_i(s,t)\right|\right]\\
	\leq&\mathbb{E}\left[\left|b_i(s,X(s))\diamond\mathcal E_i^{K}(s,t)\right|\right]+\mathbb{E}\left[\left|b_i(s,X(s))\diamond\mathcal E_i(s,t)\right|\right]\\
	\leq&\mathbb{E}\left[|b_i(s,X(s))|\diamond\mathcal E_i^{K}(s,t)\right]+\mathbb{E}\left[|b_i(s,X(s))|\diamond\mathcal E_i(s,t)\right]\\
	=&\mathbb{E}\left[|b_i(s,X(s))|\right]+\mathbb{E}\left[|b_i(s,X(s))|\right]\\
	\leq&2M.
\end{align*}
We proceed by proving that 
\begin{align*}
	\lim_{K\to\infty}\mathbb{E}\left[\left|b_i(s,X(s))\diamond\left(\mathcal E_i^{K}(s,t)-\mathcal E_i(s,t)\right)\right|\right]=0.
\end{align*} 
Let us rewrite the expected value as follows:
\begin{align*}
	&\mathbb{E}\left[\left|b_i(s,X(s))\diamond\left(\mathcal E_i^{K}(s,t)-\mathcal E_i(s,t)\right)\right|\right]=\mathbb{E}\left[\left|b_i(s,X(s))\diamond\mathcal E_i^{K}(s,t)-b_i(s,X(s))\diamond\mathcal E_i(s,t)\right|\right]
\end{align*}
Using \eqref{eq:gjessing} we get rid of the \textit{Wick product} and write
\begin{align*}
&\quad=\mathbb{E}\left[\left|\mathtt{T}_{-\Phi[\sigma^K(s,t)]}b_i(s,X(s))\mathcal E_i^{K}(s,t)-\mathtt{T}_{-\Phi[\sigma_i(s,t)]}b_i(s,X(s))\mathcal E_i(s,t)\right|\right].
\end{align*}
Adding and subtracting $\mathtt{T}_{-\Phi[\sigma_i(s,t)]}b_i(s,X(s))\mathcal E_i^{K}(s,t)$ inside the absolute value and then using the triangular inequality yields
\begin{align*}
&\quad\leq\mathbb{E}\left[\left|\mathtt{T}_{-\Phi[\sigma^K(s,t)]}b_i(s,X(s))\mathcal E_i^{K}(s,t)-\mathtt{T}_{-\Phi[\sigma_i(s,t)]}b_i(s,X(s))\mathcal E_i^{K}(s,t)\right|\right]\\
&\quad\quad+\mathbb{E}\left[\left|\mathtt{T}_{-\Phi[\sigma_i(s,t)]}b_i(s,X(s))\mathcal E_i^{K}(s,t)
-\mathtt{T}_{-\Phi[\sigma_i(s,t)]}b_i(s,X(s))\mathcal E_i(s,t)\right|\right]\\
&\quad=\mathbb{E}\left[|b_i(s,\mathtt{T}_{-\Phi[\sigma_i^K(s,t)]}X(s))-b_i(s,\mathtt{T}_{-\Phi[\sigma_i(s,t)]}X(s))|\mathcal E_i^{K}(s,t)\right]\\
&\quad\quad+\mathbb{E}\left[|b_i(s,\mathtt{T}_{-\Phi[\sigma_i(s,t)]}X(s))||\mathcal E_i^{K}(s,t)
-\mathcal E_i(s,t)|\right]\\
&\quad\leq L\mathbb{E}\left[|\mathtt{T}_{-\Phi[\sigma_i^K(s,t)]}X(s)-\mathtt{T}_{-\Phi[\sigma_i(s,t)]}X(s)|\mathcal E_i^{K}(s,t)\right]\\
&\quad\quad+M\mathbb{E}\left[|\mathcal E_i^{K}(s,t)
-\mathcal E_i(s,t)|\right].
\end{align*} 
Hence,
\begin{align*}
	\lim_{K\to\infty}\mathbb{E}\left[\left|b_i(s,X(s))\diamond\left(\mathcal E_i^{K}(s,t)-\mathcal E_i(s,t)\right)\right|\right]&\leq\lim_{K\to\infty}L\mathbb{E}\left[|\mathtt{T}_{-\Phi[\sigma_i^K(s,t)]}X(s)-\mathtt{T}_{-\Phi[\sigma_i(s,t)]}X(s)|\mathcal E_i^{K}(s,t)\right]\\
	&+\lim_{K\to\infty}M\mathbb{E}\left[|\mathcal E_i^{K}(s,t)
	-\mathcal E_i(s,t)|\right].
\end{align*} 
The second term above converges to zero by the discussion on Appendix A.

At this point  we will need the following lemma;

\begin{lemma}\label{lemma}
Let $Y\in\mathbb L^q(\Omega)$ for some $q\in (0,\infty)$ and let $\{f_n\}_{n\geq 1}$ be a sequence converging to $f$ in $\mathcal H_{\phi}$, then it holds that 
\begin{align*}
\lim_{n\to\infty} \mathtt T_{\Phi f_n}Y=\mathtt T_{\Phi f}Y, \quad\mbox{ in $\mathbb{L}^p(\Omega)$ for all $0<p<q<\infty$}
\end{align*}
\end{lemma}

\begin{proof}
For simplicity we will consider the case of random variable $Y$ depending only on a one dimensional fBm that can be seen as one of the components of our $d$-dimensional fBm, the general case does not present further difficulties.
Notice that 
\begin{align*}
	\mathtt{T}_{\Phi f_n} B_t(\omega)&=B_t^H+\int_0^t\int_0^Tf_n(s)\phi(s,r)dsdr\\
	&=B_t+\int_0^T\int_0^T\chi_{[0,t]}(r)f_n(s)\phi(s,r)dsdr\\
	&=B_t+\langle f_n,\chi_{[0,t]}\rangle_{\phi},
\end{align*}
at this point we use the fact that convergence in norm implies the weak convergence, and hence if $f_n$ converges in $\mathcal H_{\phi}$ to $f$ as $n\to\infty$ we have that $\langle f_n,\chi_{[0,t]}\rangle_{\phi}$ converges to $\langle f,\chi_{[0,t]}\rangle_{\phi}$.

This implies that 
\begin{align*}
	\mathbb E[|\mathtt{T}_{\Phi f} B_t-\mathtt{T}_{\Phi f} B_t|^p]=|\langle f_n,\chi_{[0,t]}\rangle_{\phi}-\langle f,\chi_{[0,t]}\rangle_{\phi}|^p\to 0,\; \text{as } n\to\infty.
\end{align*}
Furthermore notice that this holds for any random variable in the \textit{Gaussian Hilbert space } (e.g. \cite{janson1997gaussian})
\begin{align*}
	\mathcal G_{\phi}:=\bigg\{\int_0^Tg(s)dB_s;\; g\in\mathcal H_{\phi}\bigg\}.
\end{align*}
At this point if $Y\in\mathbb L^q(\Omega)$ we have that for any $\epsilon>0$ there's a polynomial random variable $P$ (which is a polynomial in some random variables in $\mathcal G_{\phi}$) such that $\|Y-P\|_q<\epsilon$ (the existence of such a random variable is guaranteed by \cite[Theorem 3.2]{hu2003fractional} together with  \cite[Theorem 2.11]{janson1997gaussian}).
By the triangle inequality we have that for $0<p<q$
\begin{align*}
	\|\mathtt{T}_{\Phi f_n}X-\mathtt{T}_{\Phi f}X\|_p&\leq \|\mathtt{T}_{\Phi f_n}X-\mathtt{T}_{\Phi f_n}P\|_p +\|\mathtt{T}_{\Phi f_n}P-\mathtt{T}_{\Phi f}P\|_p+\|\mathtt{T}_{\Phi f}P-\mathtt{T}_{\Phi f}Y\|_p\\
\end{align*}
Now using the fractional Girsanov's theorem we have
\begin{align*}
\|\mathtt{T}_{\Phi f}P-\mathtt{T}_{\Phi f}Y\|_p&=\mathbb{E}\left[|P-Y|^p\mathcal E(f)\right]^{1/p}\\
&\leq \mathbb E\left[|P-Y|^{pp_1}\right]^{1/(pp_1)}\mathbb{E}\left[\mathcal{E}(f)^{p_2/p}\right]^{1/(pp_2)}\\
&=\|P-Y\|_q\|\|\mathcal{E}(f)^{1/p^2}\|_r
\end{align*}
where $q:=p_1p$, $r:=p_2p$ and $\frac{1}{p_1}+\frac{1}{p_2}=1$. Same happens with $\|\mathtt{T}_{\Phi f_n}P-\mathtt{T}_{\Phi f_n}Y\|_p$.
At this point we notice that
\begin{align*}
\|\mathcal E(f_n)^{1/p^2}\|_r&\leq \mathbb E\left[\exp\bigg\{r/p^2\int_0^T f_n(s)dB_s^H\bigg\}\right]^{1/r}\\
& \leq  \sup_{n}\exp\bigg\{\frac{r}{2p^4}|f_n|_{\phi}^2\bigg\}=:C.
\end{align*}
It follows that 
\begin{align*}
\|\mathtt{T}_{\Phi f_n}X-\mathtt{T}_{\Phi f}X\|_p	&\leq 2C \|P-Y\|_q+\|\mathtt{T}_{f_n}P-\mathtt{T}_{\Phi f}P\|_p\\
&\leq (2C+1)\epsilon
\end{align*}
provided $n$ is large enough; since $\epsilon$ was arbitrary the proof is complete.
\end{proof}
From remark \ref{remark}, lemma \ref{lemma}  and the fact that 
\begin{align*}
	\lim_{K\to\infty}\mathcal E_i^{K}(s,t)=\mathcal E_i(s,t),\quad\mbox{ in $\mathbb{L}^p(\Omega)$ for all $p\geq 1$} 
\end{align*}
it follows that 
\begin{align*}
	\lim_{K\to\infty}L\mathbb{E}\left[|\mathtt{T}_{-\Phi\sigma_i^K(s,t)}X(s)-\mathtt{T}_{-\Phi\sigma_i(s,t)}X(s)|\mathcal E_i^{K}(s,t)\right]=0,
\end{align*}
completing the proof.

\section{Proof theorem \ref{th:fokker planck}}
Let $\varphi\in C_0^2([0,T]\times\mathbb R^d)$, i.e. a two times continuously differentiable function on $[0,T]\times\mathbb R^d$ with compact support, and in order to ease the notation we set $\mathbf{z}:=\mathbf{x}\rvert_{{x}_{ki}=Z_k^{(i)}}$.

Then by \ref{eq:solution def} we have
\begin{align*}
0&=\varphi(T,X^K(T))-\varphi(0,c)\\
&=\int_0^T\left[\partial_t\varphi(r,u(r,\mathbf{z}))+\sum_{i=1}^{d}\partial_i\varphi(r,u(r,\mathbf{z}))\partial_tu_i(r,\mathbf{z})\right]dr\\
&=\int_0^T \partial_t\varphi(r,u(r,\mathbf{z}))dr\\
&+\sum_{i=1}^{d}\int_0^T\partial_i\varphi(r,u(r,\mathbf{z}))\left(b_i(t,u(r,\mathbf{z}))+\sigma_i(t)\sum_{k=1}^K\xi_k(t)\left[x_{ki}u_i(r,\mathbf{z})-\partial_{x_{ki}}u_i(r,\mathbf{z})\right]\right)dr\\
&=\mathcal A+\mathcal B+\mathcal C+\mathcal D,
\end{align*}
where
\begin{align*}
\mathcal A&=\int_0^T \partial_t\varphi(r,u(r,\mathbf{z}))dr,\\
\mathcal B&=\int_0^T\nabla\varphi(r,u(r,\mathbf{z}))\bullet b(s,u(r,\mathbf{z})) dr,\\
\mathcal C&=\sum_{i=1}^{d}\sum_{k=1}^K\int_0^T\partial_i\varphi(r,u(r,\mathbf{z}))\sigma_i(r)\xi_k(r)[x_{ki}u_i(r,\mathbf{z})]dr,\\
\mathcal D&=-\sum_{i=1}^{d}\sum_{k=1}^K\int_0^T\partial_i\varphi(r,u(r,\mathbf{z}))\sigma_i(r)\xi_k(r)\partial_{ik}u_i(r,\mathbf{z})dr,
\end{align*}
where $\bullet$ denotes the inner product in $\mathbb R^d$.
Taking expectation to the first and last term above we obtain
\begin{align*}
0=\mathbb{E}[\mathcal A]+\mathbb{E}[\mathcal B]+\mathbb{E}[\mathcal C]+\mathbb{E}[\mathcal D].
\end{align*}
Now using the fact that $\mathbf{z}$ is a standard Gaussian matrix where the components are mutually independent, we have 
\begin{align*}
\mathbb E[\mathcal C]&=\sum_{i=1}^{d}\sum_{k=1}^K\int_0^T\sigma_i(r)\xi_k(r)\int_{\mathbb R^{K\times d}}\partial_i\varphi(r,u(r,\mathbf{x}))u_i(r,\mathbf{x})x_{ki} (2 \pi)^{-K\times d/2}e^{-\frac 1 2\|\mathbf{x}\|_F^2}d\mathbf{x}dr,
\end{align*}
integration by parts yields
\begin{align*}
&=\sum_{i=1}^{d}\sum_{k=1}^K\int_0^T\sigma_i(r)\xi_k(r)\int_{\mathbb R^{K\times d}}\partial_i\varphi(r,u(r,\mathbf{x}))u_i(r,\mathbf{x}) (2 \pi)^{-K\times d/2}\partial_{x_{ki}}e^{-\frac 1 2\|\mathbf{x}\|_F^2}d\mathbf{x}dr\\
&=\sum_{i,j=1}^{d}\sum_{k=1}^K\int_0^T\sigma_i(r)\xi_k(r)\int_{\mathbb R^{K\times d}}\partial_j\partial_i\varphi(r,u(r,\mathbf{x}))\partial_{x_{ki}}u_j(r,\mathbf{x})u_i(r,\mathbf{x}) (2 \pi)^{-K\times d/2}\partial_{x_{ki}}e^{-\frac 1 2\|\mathbf{x}\|_F^2}d\mathbf{x}dr\\
&+\sum_{i=1}^{d}\sum_{k=1}^K\int_0^T\sigma_i(r)\xi_k(r)\int_{\mathbb R^{K\times d}}\partial_i\varphi(r,u(r,\mathbf x))\partial_{x_{ki}}u_i(r,\mathbf{x})(2 \pi)^{-K\times d/2}\partial_{x_{ki}}e^{-\frac 1 2\|\mathbf{x}\|_F^2}d\mathbf{x}dr\\
&=\mathbb E\left[\sum_{i,j=1}^{d}\sum_{k=1}^K\int_0^T\sigma_i(r)\xi_k(r)\partial_j\partial_i\varphi(r,u(r,\mathbf{z}))\partial_{x_{ki}}u_j(r,\mathbf{z})u_i(r,\mathbf{z})dr\right]\\
&+\mathbb E\left[\sum_{i=1}^{d}\sum_{k=1}^K\int_0^T \partial_i\varphi(r,u(r,\mathbf z))\sigma_i(r)\xi_k(r)\partial_{x_{ki}}u_i(r,\mathbf{z})dr\right]
\end{align*}
and now notice that the last term above equals $-\mathbb E\left[\mathcal D\right]$.

At this point we have that 
\begin{align*}
0&=\mathbb E[\mathcal A]+\mathbb E[\mathcal B]+\mathbb E\left[\sum_{i,j=1}^{d}\sum_{k=1}^K\int_0^T\sigma_i(r)\xi_k(r)\partial_j\partial_i\varphi(r,u(r,\mathbf{z}))\partial_{x_{ki}}u_j(r,\mathbf{z})u_i(r,\mathbf{z})dr\right].
\end{align*}

Using Tower's property yields
\begin{align*}
0&=\mathbb E[\mathcal A]+\mathbb E[\mathcal B]+\mathbb E\left[\sum_{i,j=1}^{d}\sum_{k=1}^K\int_0^T\sigma_i(r)\xi_k(r)\partial_j\partial_i\varphi(r,u(r,\mathbf{z}))u_i(r,\mathbf{z})\mathbb E[\partial_{x_{ki}}u_j(r,\mathbf{z})|\mathcal G_i(r)]dr\right]\\
&=\mathbb E[\mathcal A]+\mathbb E[\mathcal B]+\mathbb E\left[\sum_{i,j=1}^{d}\sum_{k=1}^K\int_0^T\sigma_i(r)\xi_k(r)\partial_j\partial_i\varphi(r,u(r,\mathbf{z}))u_i(r,\mathbf{z})g_{ki}^{(j)}(r,u_i(r,\mathbf{z}))dr\right]
\end{align*}

where $\mathcal G_i(r)$ is the sigma algebra generated by the random variable $u_i(r,\mathbf{z})$, and the function $g_{ki}^{(j)}:[0,T]\times\mathbb R^{K\times d}$ is a measurable function, whose existence is guaranteed by the Doob's lemma chosen to satisfy
\begin{align}\label{eq:g}
g_{ki}^{(j)}(r,u_i(r,\mathbf{z}))=\mathbb E[\partial_{x_{ki}}u_j(r,\mathbf{z})|\mathcal G_i(r)].
\end{align}

Putting everything together and using (\ref{eq:solution def}) we obtain

\begin{align*}
0&=\mathbb E\bigg[\int_0^T \partial_t\varphi(r,X^K(r))dr+\int_0^T\nabla\varphi(r,X^K(r))\bullet b(s,X^K(r)) dr\\
&+\sum_{i,j=1}^{d}\sum_{k=1}^K\int_0^T\sigma_i(r)\xi_k(r)\partial_j\partial_i\varphi(r,X^K(r))X_i^k(r)g_{ki}^{(j)}(r,X_i^K(r))dr\bigg]
\end{align*}

Observe that the last member above contains expectations of functions of the random vector $X^K(r)$, for $r\in[0,T]$; therefore, writing the law of this random vector as
\begin{align*}
	\mu^{K}(r,A):=\mathbb P(\{\omega\in\Omega:X^K(r,\omega)\in A\}), \quad r\in [0,T], A\in\mathcal B(\mathbb R^d)
\end{align*}
\begin{align*}
0&=\int_0^T\int_{\mathbb R^d}\bigg[\partial_t\varphi(r,x)+\nabla\varphi(r,x)\bullet b(s,x)\\
&+ \sum_{i,j=1}^{d}\sum_{k=1}^K\int_0^T\sigma_i(r)\xi_k(r)\partial_j\partial_i\varphi(r,x)x_ig_{ki}^{(j)}(r,x_i)dr\bigg]d\mu^K(r,x)dr.
\end{align*}
The last equalities hold for any test function $\varphi\in C_0^2([0,T]\times\mathbb R^d)$ and this completes the proof.

\section*{Appendix A}\label{appendix}
Fix $s,t\in [0,T]$ and without loss of generality assume that $t\geq s$, then  using the basic inequality $|e^{X}-e^{Y}|\leq |e^{X}+e^Y|\cdot|X-Y|$ it holds that
\begin{align*}
|\mathcal E_i^K(s,t)-\mathcal E_i(s,t)|&\leq |\mathcal E_i^K(s,t)+\mathcal E_i(s,t)|\\
&\times\left|\int_0^T\sigma_i^K(t,s;q)dB_i(q)-\frac 1 2 |\sigma_i^K(t,s;\cdot)|_{\phi}^2-\int_0^t\sigma_i(s)dB_i(s)+\frac 1 2|\chi_{[0,t]}\sigma_i|_{\phi}^2\right|\\
&\leq |\mathcal E_i^K(s,t)+\mathcal E_i(s,t)|\\
&\times\left(\left|\int_0^T[\sigma_i^K(t,s;q)-\sigma_i(q)]dB_i(q) \right|+\frac 1 2\left||\sigma_i^K(t,s;\cdot)|_{\phi}^2-|\chi_{[s,t]}\sigma_i|_{\phi}^2\right|\right).
\end{align*}

Now let's write
\begin{align*}
\left||\sigma_i^K(t,s;\cdot)|_{\phi}^2-|\chi_{[s,t]}\sigma_i|_{\phi}^2\right|&\leq \left|(|\sigma_i^K(t,s;\cdot)|_{\phi}+|\chi_{[s,t]}\sigma_i|_{\phi})(|\sigma_i^K(t,s;\cdot)|_{\phi}-|\chi_{[s,t]}\sigma_i|_{\phi})\right|\\
&\leq S^2 T^{2H}|\sigma_i^K(t,s)-\chi_{[s,t]}\sigma_i|_{\phi}
\end{align*}
where we used the triangular inequality and $S$ is a constant such that $|\sigma_i(t)|\leq S$ for all $t\in [0,T]$.

At this point we raise both sides to the $p\geq 1$ and take expectation yielding
\begin{align*}
\mathbb E\left[|\mathcal E_i^K(s,t)-\mathcal E_i(s,t)|^p\right]&\leq 2^{p-1}\mathbb E\bigg[|\mathcal E_i^K(s,t)+\mathcal E_i(s,t)|^p\bigg(|I(\sigma_i^K(t,s)-\chi_{[s,t]}\sigma_i)|^p\\
&+S^{2p} T^{2Hp}|\sigma_i^K(t,s)-\chi_{[s,t]}\sigma_i|_{\phi}^p\bigg)\bigg]
\end{align*}
where $I(\cdot)$ denotes the fractional Wiener integral.
Using  Hölder's inequality where $\frac{1}{p_1}+\frac{1}{p_2}=\frac 1 p$ we have
\begin{align*}
&\leq 2^{p-1}\bigg\{\|\mathcal E_i^K(s,t)+\mathcal E_i(s,t)\|_{{p_1}}^{p}\|I(\sigma_i^K(t,s)-\chi_{[s,t]}\sigma_i)\|_{{p_2}}^{p}\\
&+S^{2p} T^{2Hp}\|\mathcal E_i^K(s,t)+\mathcal E_i(s,t)\|_{{p}}^{p}|\sigma_i^K(t,s)-\chi_{[s,t]}\sigma_i|_{\phi}^p\bigg\}\\
&\leq 2^{p-1}\bigg\{\left(2^{p}e^{p(p_1-1)/2|\chi_{[s,t]}\sigma_i|_{\phi}^2}\right)2^{p/2}\Gamma(p_2+1)^{p/p_2}/\sqrt{\pi}|\sigma_i^K(t,s)-\chi_{[s,t]}\sigma_i|_{\phi}^p\\
&+S^{2p} T^{2Hp}\left(2^pe^{p(p-1)/2|\chi_{[s,t]}\sigma_i|_{\phi}^2}\right)|\sigma_i^K(t,s)-\chi_{[s,t]}\sigma_i|_{\phi}^p\bigg\}
\end{align*}

and hence 
\begin{align*}
\mathbb E\left[|\mathcal E_i^K(s,t)-\mathcal E_i(s,t)|^p\right]\leq \mathtt C |\sigma_i^K(t,s)-\chi_{[s,t]}\sigma_i|_{\phi}^p\to 0,
\end{align*}
as $K\to\infty$ where
\begin{align*}
\mathtt C=2^{3p/2-1}e^{p(p_1-1)/2|\chi_{[s,t]}\sigma_i|_{\phi}^2}\Gamma(p_2+1)^{p/p_2}/\sqrt{\pi}+2^{2p-1}S^{2p} T^{2Hp}e^{p(p-1)/2|\chi_{[s,t]}\sigma_i|_{\phi}^2}.
\end{align*}

\bibliographystyle{IEEEtran}
\bibliography{bib}

\begin{thebibliography}{10}
\providecommand{\url}[1]{#1}
\csname url@samestyle\endcsname
\providecommand{\newblock}{\relax}
\providecommand{\bibinfo}[2]{#2}
\providecommand{\BIBentrySTDinterwordspacing}{\spaceskip=0pt\relax}
\providecommand{\BIBentryALTinterwordstretchfactor}{4}
\providecommand{\BIBentryALTinterwordspacing}{\spaceskip=\fontdimen2\font plus
\BIBentryALTinterwordstretchfactor\fontdimen3\font minus
  \fontdimen4\font\relax}
\providecommand{\BIBforeignlanguage}[2]{{%
\expandafter\ifx\csname l@#1\endcsname\relax
\typeout{** WARNING: IEEEtran.bst: No hyphenation pattern has been}%
\typeout{** loaded for the language `#1'. Using the pattern for}%
\typeout{** the default language instead.}%
\else
\language=\csname l@#1\endcsname
\fi
#2}}
\providecommand{\BIBdecl}{\relax}
\BIBdecl

\bibitem{wong1965convergence}
E.~Wong and M.~Zakai, ``On the convergence of ordinary integrals to stochastic
  integrals,'' \emph{The Annals of Mathematical Statistics}, vol.~36, no.~5,
  pp. 1560--1564, 1965.

\bibitem{stroock2020support}
D.~W. Stroock and S.~R. Varadhan, ``On the support of diffusion processes with
  applications to the strong maximum principle,'' in \emph{Contributions to
  Probability Theory}.\hskip 1em plus 0.5em minus 0.4em\relax University of
  California Press, 2020, pp. 333--360.

\bibitem{hu1996wick}
Y.~Hu and B.~{\O}ksendal, ``Wick approximation of quasilinear stochastic
  differential equations,'' in \emph{Stochastic Analysis and Related Topics
  V}.\hskip 1em plus 0.5em minus 0.4em\relax Springer, 1996, pp. 203--231.

\bibitem{LanconelliScorolli2021}
\BIBentryALTinterwordspacing
A.~Lanconelli and R.~Scorolli, ``Wong–zakai approximations for quasilinear
  systems of itô’s type stochastic differential equations,''
  \emph{Stochastic Processes and their Applications}, vol. 141, pp. 57--78,
  2021. [Online]. Available:
  \url{https://www.sciencedirect.com/science/article/pii/S0304414921001198}
\BIBentrySTDinterwordspacing

\bibitem{hu2018ito}
Y.~Hu, ``It{\^o} type stochastic differential equations driven by fractional
  brownian motions of hurst parameter,'' \emph{Stochastics}, vol.~90, no.~5,
  pp. 720--761, 2018.

\bibitem{hu2003fractional}
Y.~Hu and B.~{\O}ksendal, ``Fractional white noise calculus and applications to
  finance,'' \emph{Infinite dimensional analysis, quantum probability and
  related topics}, vol.~6, no.~01, pp. 1--32, 2003.

\bibitem{holden1996stochastic}
H.~Holden, B.~{\O}ksendal, J.~Ub{\o}e, and T.~Zhang, ``Stochastic partial
  differential equations,'' in \emph{Stochastic partial differential
  equations}.\hskip 1em plus 0.5em minus 0.4em\relax Springer, 1996, pp.
  141--191.

\bibitem{hu2016analysis}
Y.~Hu, \emph{Analysis on Gaussian spaces}.\hskip 1em plus 0.5em minus
  0.4em\relax World Scientific, 2016.

\bibitem{janson1997gaussian}
S.~Janson \emph{et~al.}, \emph{Gaussian hilbert spaces}.\hskip 1em plus 0.5em
  minus 0.4em\relax Cambridge university press, 1997, no. 129.

\bibitem{duncan2000stochastic}
T.~E. Duncan, Y.~Hu, and B.~Pasik-Duncan, ``Stochastic calculus for fractional
  brownian motion i. theory,'' \emph{SIAM Journal on Control and Optimization},
  vol.~38, no.~2, pp. 582--612, 2000.

\bibitem{taylor2013partial}
M.~Taylor, \emph{Partial differential equations III: Nonlinear
  Equations}.\hskip 1em plus 0.5em minus 0.4em\relax in: Applied Mathematical
  Sciences, Springer, New York, 1997, vol. 117.

\bibitem{bressan2000hyperbolic}
A.~Bressan, \emph{Hyperbolic systems of conservation laws: the one-dimensional
  Cauchy problem}.\hskip 1em plus 0.5em minus 0.4em\relax Oxford University
  Press on Demand, 2000, vol.~20.

\bibitem{majda2012compressible}
A.~Majda, \emph{Compressible fluid flow and systems of conservation laws in
  several space variables}.\hskip 1em plus 0.5em minus 0.4em\relax Springer
  Science \& Business Media, 2012, vol.~53.

\bibitem{biagini2008stochastic}
F.~Biagini, Y.~Hu, B.~{\O}ksendal, and T.~Zhang, \emph{Stochastic calculus for
  fractional Brownian motion and applications}.\hskip 1em plus 0.5em minus
  0.4em\relax Springer Science \& Business Media, 2008.

\end{thebibliography}

\end{document}